\documentclass[12pt]{amsart}
\usepackage{amssymb}
\usepackage{hyperref}

\newtheorem{thm}{Theorem}[section]
\newtheorem{lem}[thm]{Lemma}
\newtheorem{cor}[thm]{Corollary}
\newtheorem{prop}[thm]{Proposition}
\newtheorem{ex}[thm]{Example}

\newtheorem{qu}[thm]{Question}

\newtheorem{con}[thm]{Conjecture}
\newtheorem*{prob*}{Open problem}

\theoremstyle{definition}

\newtheorem{defi}[thm]{Definition}

\theoremstyle{remark}

\newtheorem{rem}[thm]{Remark}
\newtheorem*{rem*}{Remark}

\DeclareMathOperator{\id}{id}

\DeclareMathOperator{\rad}{rad}
\DeclareMathOperator{\Hom}{Hom}
\DeclareMathOperator{\Aff}{Aff}
\DeclareMathOperator{\Aut}{Aut}

\newcommand{\kringel}{\mathbin{\raise1pt\hbox{$\scriptstyle\circ$}}}
\newcommand{\pkt}{\mathbin{\raise0pt\hbox{$\scriptstyle\bullet$}}}

\newcommand{\C}{\mathbb{C}}

\newcommand{\R}{\mathbb{R}}

\newcommand{\tr}{\mathop{\rm tr}}
\newcommand{\ad}{{\rm ad}}

\newcommand{\End}{{\rm End}}
\newcommand{\Der}{{\rm Der}}
\newcommand{\Isom}{{\rm Isom}}

\newcommand{\La}{\mathfrak{a}}
\newcommand{\Lb}{\mathfrak{b}}

\newcommand{\Ld}{\mathfrak{d}}

\newcommand{\Lg}{\mathfrak{g}}
\newcommand{\Lh}{\mathfrak{h}}

\newcommand{\Ll}{\mathfrak{l}}
\newcommand{\Ln}{\mathfrak{n}}

\newcommand{\Lr}{\mathfrak{r}}
\newcommand{\Ls}{\mathfrak{s}}

\newcommand{\CD}{\mathcal{D}}

\newcommand{\im}{\mathop{\rm im}}

\newcommand{\al}{\alpha}
\newcommand{\be}{\beta}
\newcommand{\ga}{\gamma}
\newcommand{\de}{\delta}
\newcommand{\ep}{\varepsilon}
\newcommand{\ka}{\kappa}
\newcommand{\la}{\lambda}

\newcommand{\ra}{\rightarrow}

\renewcommand{\phi}{\varphi}

\begin{document}


\title[Post-Lie Algebra Structures]{Post-Lie algebra structures on pairs of Lie algebras}

\author[D. Burde]{Dietrich Burde}
\author[K. Dekimpe]{Karel Dekimpe}
\address{Fakult\"at f\"ur Mathematik\\
Universit\"at Wien\\
  Oskar-Morgenstern-Platz 1\\
  1090 Wien \\
  Austria}
\email{dietrich.burde@univie.ac.at}
\address{Katholieke Universiteit Leuven\\
Campus Kortrijk\\
8500 Kortrijk\\
Belgium}
\email{karel.dekimpe@kuleuven-kortrijk.be}

\date{\today}

\thanks{The first author acknowledges support by the Austrian Science Foundation FWF, grant P28079.}
\thanks{The second author expresses his gratitude towards the Erwin Schr\"odinger International
Institute for Mathematical Physics}

\begin{abstract}
We study post-Lie algebra structures on pairs of Lie algebras $(\Lg,\Ln)$, which describe simply transitive
nil-affine actions of Lie groups. We prove existence results for such structures depending on the 
interplay of the algebraic structures of $\Lg$ and $\Ln$. We consider the classes of
simple, semisimple, reductive, perfect, solvable, nilpotent, abelian and unimodular Lie algebras.
Furthermore we consider commutative post-Lie algebra structures on perfect Lie algebras. Using
Lie algebra cohomology we can classify such structures in several cases. We also study 
commutative structures on low-dimensional Lie algebras and on nilpotent Lie algebras.
\end{abstract}

\maketitle

\section{Introduction}

Post-Lie algebras and post-Lie algebra structures are an important generalization of 
left-symmetric algebras (also called pre-Lie algebras) and left-symmetric algebra structures on Lie algebras,
which arise in many areas of algebra and geo\-metry \cite{BU24}, such as 
left-invariant affine structures on Lie groups, affine crystallographic groups, simply transitive affine 
actions on Lie groups, convex homogeneous cones, faithful linear representations of Lie algebras, 
operad theory and several other areas. A particular interesting problem with a long history concerns
crystallographic groups and crystallographic structures on groups. Here a Euclidean crystallographic structure
on a group $\Gamma$ is a representation $\rho\colon \Gamma\ra \Isom(\R^n)$ letting $\Gamma$ act properly
discontinuously and cocompactly on $\R^n$. By the Bieberbach theorems, the groups admitting a Euclidean
crystallographic structure are precisely the finitely generated virtually abelian groups. The torsion-free
crystallographic groups, called Bieberbach groups, are exactly the fundamental groups of compact flat Riemannian 
manifolds. John Milnor addressed the question how this result can be generalized to
{\it affine crystallographic structures} for $\Gamma$, i.e., for representations $\rho\colon \Gamma\ra \Aff(\R^n)$
letting $\Gamma$ act  properly discontinuously and cocompactly on $\R^n$. More precisely he asked whether or not 
every virtually polycyclic group admits an affine crystallographic structure. 
Although there was a lot of evidence for Milnor's question having a positive answer, counterexamples were 
found \cite{BENO,BU5}, in particular in terms of left-symmetric structures on Lie algebras. For the history and references see 
\cite{BU24}.
Another natural question then was whether one could find a reasonable class of generalized affine structures on $\Gamma$ such that
Milnor's question would have a positive answer. This was indeed possible and lead to the notion of a so-called
{\it nil-affine crystallographic structure} for $\Gamma$, i.e., a representation $\rho\colon \Gamma\ra \Aff(N)$,
letting $\Gamma$ act  properly discontinuously and cocompactly on a simply connected, connected nilpotent Lie group $N$.
Here $\Aff(N)=N\rtimes \Aut(N)$ denotes the affine group of $N$. It was shown in \cite{DE3} and \cite{BA2} that
every virtually polycyclic group admits a nil-affine crystallographic structure. The natural generalization
of left-symmetric structures, which correspond to the case $N=\R^n$,  are exactly the post-Lie algebra structures on
pairs of Lie algebras. This is the reason that these structures are so fundamental for nil-affine crystallographic
structures. In the nilpotent case, these structures can also be formulated by nil-affine actions of a nilpotent
Lie group $G$ on another nilpotent Lie group $N$, and such Lie group actions can be translated to the level of Lie algebras
leading directly to post-Lie algebra structures \cite{BU41}. \\
As in the case of left-symmetric structures, the existence question of post-Lie algebra structures
is very important (and very hard in general). We already studied an intermediate case, the so-called LR-structures on
Lie algebras in \cite{BU34, BU38}, and have obtained first results for the general case in \cite{BU41, BU44}.
In this paper we are able to prove more general results. One result concerns perfect Lie algebras $\Lg$, i.e.,
satisfying $[\Lg,\Lg]=\Lg$. It has been proved in \cite{SEG} that a perfect Lie algebra over a field of
characteristic zero does not admit a left-symmetric structure. We extend this result and show that
if $\Lg$ is perfect and $\Ln$ is solvable and non-nilpotent, then there is no post-Lie algebra structure 
on $(\Lg,\Ln)$. We will also prove that if $(\Lg,\Ln)$ is a pair of Lie algebras, where $\Lg$ is simple, and $\Ln$ is 
not isomorphic to $\Lg$, then there is no post-Lie algebra structure on $(\Lg,\Ln)$.
This shows that the assumption that $\Lg$ is simple is very strong. Already the case that $\Ln$ is semisimple
has strong implications. For example, there are no  post-Lie algebra structure on $(\Lg,\Ln)$ with $\Ln$ semisimple
and $\Lg$ solvable. \\
In the last part of the paper we study commutative post-Lie algebra structures. A post-Lie algebra structure 
on a pair $(\Lg,\Ln)$ is called {\it commutative}, if  the Lie algebras $\Lg$ and $\Ln$ coincide, so that the product 
is commutative. In this case we just write $\Lg$ instead of the pair $(\Lg,\Lg)$. In this special case one could
perhaps hope for a good classification result. We first show that any commutative 
post-Lie algebra structure on a semisimple Lie algebra is trivial, i.e., is given by the zero product. Then we consider 
commutative post-Lie algebra structures on perfect Lie algebras. This involves the study of the outer derivation 
algebra of a perfect Lie algebra, and more generally its Lie algebra cohomology. It is known that 
cohomological questions for perfect Lie algebras can be difficult, e.g., see \cite{BEN}, \cite{PIR}.
We prove that commutative post-Lie algebra structures on perfect Lie algebras are trivial in several cases, and
conjecture that this is always the case. We also classify commutative post-Lie algebra structures in low dimension. \\
Finally we mention that post-Lie algebra structures have become more popular also in other areas. As an example,
post-Lie algebras have been studied by Vallette and Loday in connection with homology of partition posets and 
the study of Koszul operads \cite{VAL}, \cite{LOD}. Also, post-Lie algebras have been studied in connection
with isospectral flows, Yang-Baxter equations, Lie-Butcher Series and Moving Frames \cite{ELM}.

\section{Preliminaries}

All Lie algebras are assumed to be finite-dimensional over a field $k$. For the structure results
we usually assume that $k$ is the field of complex numbers.
Post-Lie algebra structures on pairs of Lie algebras $(\Lg,\Ln)$ are defined as
follows, see \cite{BU41}:

\begin{defi}\label{pls}
Let $\Lg=(V, [\, ,])$ and $\Ln=(V, \{\, ,\})$ be two Lie brackets on a vector space $V$.
A {\it post-Lie algebra structure} on the pair $(\Lg,\Ln)$ is a $k$-bilinear product
$x\cdot y$ satisfying the identities:
\begin{align}
x\cdot y -y\cdot x & = [x,y]-\{x,y\} \label{post1}\\
[x,y]\cdot z & = x\cdot (y\cdot z) -y\cdot (x\cdot z) \label{post2}\\
x\cdot \{y,z\} & = \{x\cdot y,z\}+\{y,x\cdot z\} \label{post3}
\end{align}
for all $x,y,z \in V$.
\end{defi}

We denote by $L(x)$ the left multiplication operators of the algebra $A=(V,\cdot)$, and by $R(x)$ the right
multiplication operators. So we have $L(x)(y)=x\cdot y$ and $R(x)(y)=y\cdot x$ for all 
$x,y\in V$. By \eqref{post3}, all $L(x)$ are derivations of the Lie algebra $(V,\{,\})$. Moreover
by \eqref{post2}, the left multiplication
\[
L\colon \Lg\ra \Der(\Ln)\subseteq \End (V),\; x\mapsto L(x)
\]
is a linear representation of $\Lg$.
Suppose that $x\cdot y$ is a post-Lie algebra structure on the pair of Lie algebras
$(\Lg,\Ln)$ such that $\Ln$ is centerless and satisfies $\Der(\Ln)=\ad (\Ln)$.
Then there is a unique $\phi\in \End(V)$ such that 
\begin{align*}
x\cdot y &=\{\phi(x),y \} 
\end{align*}
for all $x,y\in V$, see Lemma $2.9$ in \cite{BU41}. This means $L(x)=\ad (\phi(x))$ for the linear operators $L(x)$.
In this case the axioms of a post-Lie algebra structure can be formulated as follows,
see \cite{BU41}:

\begin{prop}\label{2.2}
Let $\Ln$ be a semisimple Lie algebra and  $\phi\in \End(V)$. Then
the product $x\cdot y=\{\phi(x),y \}$ is a post-Lie algebra structure on $(\Lg,\Ln)$
if and only if
\begin{align*}
\{\phi(x),y\}+\{x,\phi(y)\} & =[x,y]-\{x,y\},\\
\phi([x,y]) & = \{\phi(x),\phi(y)\} 
\end{align*}
for all $x,y\in V$. 
\end{prop}
The second condition says that $\phi\colon \Lg\ra \Ln$ is a Lie algebra homomorphism. 

\section{Existence results with respect to $\Lg$}

We study the existence question of post-Lie algebra structures on pairs $(\Lg,\Ln)$, where
the algebraic structure of $\Lg$ is given. \\[0.2cm]
{\it Case 1:} Assume, that $\Lg$ is a simple Lie algebra.
It turns out  that this assumption is very strong. We will show that a post-Lie algebra 
structure on $(\Lg,\Ln)$ with $\Lg$ simple exists if and only if $\Ln$ is isomorphic to $\Lg$.  
This generalizes Proposition $4.9$ in \cite{BU41}. Furthermore, if $\Ln\cong \Lg$ then
either the product $x\cdot y$ is trivial and $[x,y]=\{x,y\}$, or we have
$[x,y]=-\{x,y\}=x\cdot y$, see Proposition $4.6$ in \cite{BU41}.

\begin{thm}\label{3.1}
Let $(\Lg,\Ln)$ be a pair of Lie algebras, where $\Lg$ is simple and $\Ln$ is arbitrary.
Suppose that $(\Lg,\Ln)$ admits a post-Lie algebra structure. Then $\Ln$ is isomorphic to $\Lg$.
\end{thm}

\begin{proof}
Denote by $\rad(\Ln)$ the solvable radical of $\Ln$.
Consider the Levi decomposition $\Ln=\Ls\ltimes \rad(\Ln)$ with a semisimple subalgebra $\Ls\cong \Ln/\rad (\Ln)$.
We have $\Ls\neq 0$ by Theorem $4.2$ of \cite{BU44}, which is a non-trivial result. For any derivation
$D\in \Der(\Ln)$ we have $D(\rad(\Ln))\subseteq \rad(\Ln)$. Hence there is a natural
projection of Lie algebras $p\colon \Der(\Ln)\ra \Der(\Ls)\cong \Ls$. Consider the composition
\[
p\circ L\colon \Lg \ra \Der(\Ls)\cong \Ls,
\]
where $L\colon \Lg\ra \Der(\Ln)$ is the representation given by the left multiplication maps $L(x)$. 
Either $p\circ L$ is injective or it is the zero map, because its kernel is an ideal in $\Lg$. 
Suppose that $p\circ L$ is injective. Then $\dim (\Lg)\le \dim (\Ls)\le \dim(\Ln)$, hence 
$\dim(\Ls)=\dim(\Ln)$. It follows that $\Ln=\Ls$, and $p\circ L\colon \Lg\ra\Ln$ is an isomorphism. 
So we are done. Otherwise we have $p\circ L=0$, so that $0=p(L(x)(y))=p(x\cdot y)$ 
for all $x,y\in \Lg$. This implies $x\cdot y\in \rad (\Ln)$ for all $x,y\in \Lg$. But then
\[
[x,y]=x\cdot y-y\cdot x-\{x,y\}
\]
implies that $\rad(\Ln)$ is also an ideal in $\Lg$. Hence either $\rad(\Ln)=0$ and $\Ln=\Ls$ as before,
or $\rad(\Ln)=\Lg$, which means $\Ls=0$, a contradiction.
\end{proof}

In other words, if $(\Lg,\Ln)$ is a pair of Lie algebras, such that $\Lg$ is simple, and $\Ln$ is a Lie algebra
which is not isomorphic to $\Lg$, then there is no post-Lie algebra structure on $(\Lg,\Ln)$.  \\[0.2cm]
{\it Case 2:} Assume, that $\Lg$ is a semisimple Lie algebra. Then we have the following result, see 
Theorem $4.2$ of \cite{BU44}:

\begin{thm}
Let $(\Lg,\Ln)$ be a pair of Lie algebras, where $\Lg$ is semisimple and $\Ln$ is solvable.
Then there is no post-Lie algebra structure on $(\Lg,\Ln)$.
\end{thm}

\noindent {\it Case 3:} Assume, that $\Lg$ is a perfect Lie algebra, i.e., $[\Lg,\Lg]=\Lg$.

\begin{prop}
Let $(\Lg,\Ln)$ be a pair of Lie algebras over a field of characteristic zero, where $\Lg$ is perfect 
and $\Ln$ is abelian or $2$-step nilpotent. Then there is no post-Lie algebra structure on $(\Lg,\Ln)$.
\end{prop}

\begin{proof}
Assume that there exists a post-Lie algebra structure on $(\Lg,\Ln)$ with abelian $\Ln$. 
By definition this is a compatible left-symmetric structure on the perfect Lie algebra $\Lg$. 
However, it has been shown in \cite{SEG}, Theorem $1$ and Proposition $1$, that any Lie algebra 
admitting such a structure must be solvable. This is a contradiction.
For the field of complex numbers, the solvability of $\Lg$ has been already shown in \cite{HEL}, 
Proposition $20$. \\
Now assume that $\Ln$ is $2$-step nilpotent. By Proposition $4.2$ of \cite{BU41} it follows that
$\Lg$ admits a left-symmetric structure. This is impossible, as we have just seen.
\end{proof}

We recall Proposition $4.4$ of \cite{BU41}:

\begin{prop}
Let $(\Lg,\Ln)$ be a pair of Lie algebras, where $\Lg$ is perfect and $\Ln$ is solvable and non-nilpotent.
Then there is no post-Lie algebra structure on $(\Lg,\Ln)$.
\end{prop}

This naturally leads to the following question:

\begin{qu}
Let $(\Lg,\Ln)$ be a pair of Lie algebras, where $\Lg$ is perfect and $\Ln$ is nilpotent.
Does there exist a  post-Lie algebra structure on $(\Lg,\Ln)$ ?
\end{qu}

We obtain examples of perfect non-semisimple Lie algebras as follows. Let $\Ls$ be a semisimple Lie algebra
with irreducible representation $V(m)$ of dimension $m\ge 2$ and form $\Lg=\Ls\times V(m)$
with Lie brackets
\[
[(x,v),(y,w)]=([x,y],x.w-y.v)
\]
for all $x,y\in \Ls$ and $v,w\in V(m)$. Then $\Lg$ is a perfect Lie algebra with radical
$\rad(\Lg)=V(m)$ and $\Lg/\rad(\Lg)\cong \Ls$. We will write $\Lg=\Ls\ltimes V(m)$. 
The first interesting example is $\Lg=\Ls\Ll_2(\C)\ltimes V(2)=
\langle e_1,e_2,e_3\rangle \ltimes \langle e_4,e_5\rangle$ with Lie brackets
\begin{align*}
[e_1,e_2] & = e_3,\;[e_1,e_3] = -2e_1,\, [e_1,e_5] = e_4, \\
[e_2,e_3] & = 2e_2,\;[e_2,e_4 ]= e_5, \\
[e_3,e_4 ] & = e_4,\;[e_3,e_5]= -e_5.
\end{align*}

We have the following result.

\begin{prop}
Let $(\Lg,\Ln)$ be a pair of Lie algebras, where $\Lg=\Ls\Ll_2(\C)\ltimes V(2)$ and $\Ln$ is nilpotent.
Then there is no post-Lie algebra structure on $(\Lg,\Ln)$.
\end{prop}

\begin{proof}
The result is clear for $\Ln$ being abelian or $2$-step nilpotent by the above results. Note that $\dim(\Ln)=5$, because it has the same underlying vector space as $\Lg$.
If $\Ln$ is $4$-step nilpotent, then the Lie algebra $\Der(\Ln)$ is solvable. 
Assume that  $(\Lg,\Ln)$ admits a post-Lie algebra structure in this case. Then by Proposition 
$2.11$ of \cite{BU41} $\Lg$ is a subalgebra of the solvable Lie algebra $\Ln\ltimes \Der(\Ln)$. 
It follows that $\Lg$ is solvable, which is a contradiction. Up to isomorphism there is only
one nilpotent Lie algebra of dimension $5$ for which its derivation algebra is not solvable, namely
the free $3$-step nilpotent Lie algebra on $2$ generators.
Hence it remains to consider this case, and we may assume that
$\Ln$ is isomorphic to this nilpotent Lie algebra of dimension $5$, given by the brackets
\[
\{e_1,e_2\}=e_3,\; \{e_1,e_3\}=e_4,\; \{e_2,e_3\}=e_5.
\]
Denote the lower central series of $\Ln$ by $\Ln^0=\Ln$, $\Ln^i=[\Ln,\Ln^{i-1}]$ for $i\ge 1$.
Then $\Ln^1\cong \C^3$,  $\Ln^2\cong \C^2$, and $\Ln^3=0$. Furthermore, $\Ln/\Ln^1\cong \C^2$ and
$\Ln/\Ln^2\cong \Lh_1$, the $3$-dimensional Heisenberg Lie algebra.
With respect to the above basis of $\Ln$, an arbitrary derivation $D\in\Der(\Ln)$ 
is given by
\[
D=\begin{pmatrix}
\al_1 & \al_2 & 0 & 0 & 0 \\
\be_1 & \be_2 & 0 & 0 & 0 \\
\ga_1 & \ga_2 & \al_1+\be_2 & 0 & 0 \\
\de_1 & \de_2 & \ga_2 & 2\al_1+\be_2 & \al_2 \\
\ep_1 & \ep_2 & -\ga_1 & \be_1 & \al_1+2\be_2 \\
\end{pmatrix}
\]
Let $\Ld_1=\left\{\begin{pmatrix} \al_1 & \al_2 \\ \be_1 & \be_2 \end{pmatrix}\mid \al_i,\be_i \in \C\right\}$ and 
$\Ld_2=\left\{\begin{pmatrix} \al_1 & \al_2 & 0\\ \be_1 & \be_2 & 0 \\ \ga_1 & \ga_2 & \al_1+\be_2 \end{pmatrix}\mid
\al_i,\be_i,\ga_i\in\C\right\}$.
There are natural projections of $\Der(\Ln)$ to $\Ld_1$ and $\Ld_2$. This induces
epimorphisms
\[
(\Ln\rtimes \Der(\Ln))/\Ln^1 \twoheadrightarrow (\Ln/\Ln^1)\rtimes \Ld_1 \text{ and }
(\Ln\rtimes \Der(\Ln))/\Ln^2 \twoheadrightarrow (\Ln/\Ln^2)\rtimes \Ld_2.
\]
Suppose that there exists a post-Lie algebra structure on $(\Lg,\Ln)$. Then there is a monomorphism
$\rho \colon \Lg\hookrightarrow \Ln\rtimes \Der(\Ln)$ given by $x\mapsto (x,L(x))$. It induces morphisms
$\rho_i\colon \Lg\ra \Ld_i$ for $i=1,2$. 
We claim that $\im (\rho_i)\cong\Ls\Ll_2(\C)$, and hence $\ker(\rho_i)\cong\C^2$ for $i=1,2$. 
This is easy to see for $i=1$.  Indeed, the only non-trivial ideal of $\Lg$ is $V(2)\cong \C^2$, so if the image of $\rho_1$ would not be isormorphic to $\Ls\Ll_2(\C)$ it would either be 0 or isomorphic to $\Lg$. This last case is not possible because the dimension of $\Lg$ is too big, while the first case case would imply that the image of $\rho$ would be solvable, which is also not possible.

Consider now $\im (\rho_2)\subseteq \Ld_2$. Suppose that the image
is not isomorphic to $\Ls\Ll_2(\C)$. Then 
\[
\rho_2(V(2))=\left\{ \begin{pmatrix} 0 & 0 & 0 \\ 0 & 0 & 0 \\ \ga_1 & \ga_2 & 0 \end{pmatrix} \mid \ga_i\in \C\right\}
\subseteq \im (\rho_2).
\]
But then we have a morphism $\phi\colon \Lg\ra \Der(\Ln)$ such that $\phi(V(2))$ contains matrices $A$ and $B$
of the form 
\[
A= \begin{pmatrix} 
0 & 0 & 0 & 0 & 0 \\
0 & 0 & 0 & 0 & 0 \\
1 & 0  & 0 & 0 & 0 \\
\ast & \ast & 0 & 0 & 0 \\
\ast & \ast & -1 & 0 & 0 \\
\end{pmatrix} \mbox{ and } B= \begin{pmatrix} 
0 & 0 & 0 & 0 & 0 \\
0 & 0 & 0 & 0 & 0 \\
0 & 1  & 0 & 0 & 0 \\
\ast & \ast & 1 & 0 & 0 \\
\ast & \ast & 0 & 0 & 0 \\
\end{pmatrix}.
\]
The matrices $A$ and $B$ do not commute, but on the other hand, $\phi(V(2))$ must be an abelian Lie algebra, 
so we have a contradiction. \\
Now consider the composition 
\[
\rho_3\colon \Lg \ra (\Ln\rtimes \Der(\Ln))/\Ln^2 \twoheadrightarrow (\Ln/\Ln^2)\rtimes \Ld_2.
\]
We claim that $\rho_3$ must be faithful. Otherwise $\rho_3(\C^2)=0$, and the morphism 
$\Ln/\Ln^2\rtimes \Der(\Ln)\ra \Ln/\Ln^2$ would have a $3$-dimensional image. This would imply
that $\rho_3$ induces a post-Lie algebra structure on the pair $(\Ls\Ll_2(\C),\Ln/\Ln^2)$, which is
impossible by Theorem $\ref{3.1}$. \\
Because $\rho_3$ is faithful, $\rho_3(\C^2)$ is a $2$-dimensional abelian subalgebra of the Heisenberg
algebra $\Ln/\Ln^2$. Hence  $\rho_3(\C^2)$ contains the center of $\Ln/\Ln^2$, which is $\Ln^1/\Ln^2$.
But then the composition 
\[
\rho_4\colon \Lg\ra (\Ln\rtimes \Der(\Ln))/\Ln^1 \twoheadrightarrow (\Ln/\Ln^1)\rtimes \Ld_1
\]
has a $4$-dimensional image, so that $\Lg$ has a $1$-dimensional ideal. However, this is impossible, 
because $V(2)$ is a $2$-dimensional irreducible $\Ls\Ll_2(\C)$-module. This gives a contradiction.
\end{proof}

\begin{rem}
The above result cannot be easily generalized to the case of $\Ls\Ll_n(\C)\ltimes V(m)$, since the 
classification of all nilpotent Lie algebras of dimension $n^2-1+m$ with solvable automorphism group
was used in the proof, for $n=2$, $m=2$. The same problem had already appeared in connection with the 
generalized Auslander conjecture, i.e., in the proof of Theorem $4.1$ in \cite{BU19}.
\end{rem}

\section{Existence results with respect to $\Ln$}

If we assume that $\Ln$ is simple, no such strong result holds as in the case where $\Lg$ is simple.
In fact, given a semisimple Lie algebra $\Ln$, 
there exists always a solvable, non-nilpotent Lie algebra $\Lg$ such that there is a post-Lie algebra structure 
on $(\Lg,\Ln)$, see Proposition $3.1$ in \cite{BU44}. Hence $\Ln$ and $\Lg$ need not be isomorphic in general.
On the other hand, if $\Ln$ is simple and $\Lg$ is nilpotent, then there is no post-Lie algebra structure
on $(\Lg,\Ln)$. This follows from Theorem $3.3$ of \cite{BU44}, which we recall here for convenience:

\begin{thm}
Let $\Ln$ be a semisimple Lie algebra, and $\Lg$ be a solvable, unimodular Lie algebra. Then there is no 
post-Lie algebra structure on $(\Lg,\Ln)$.
\end{thm} 

One might ask, given a simple Lie algebra $\Ln$, for which Lie algebras $\Lg$ there exists a post-Lie algebra
structure on $(\Lg,\Ln)$. For example, for $\Ln=\Ls\Ll_2(\C)$ there exists a post-Lie algebra structure
on $(\Lg,\Ln)$ exactly for $\Lg=\Ls\Ll_2(\C)$, and for all $3$-dimensional solvable Lie algebras which are 
not unimodular, see Proposition $4.7$ in \cite{BU41}. For $\Ln=\Ls\Ll_3(\C)$ we may even take Lie algebras $\Lg$ 
which are neither solvable nor semisimple:

\begin{ex}
Let $\Ln=\Ls\Ll_3(\C)$, with standard basis $\{e_1,\ldots ,e_8\}$ as given in \cite{BU44}, and $\Lg$ be
the Lie algebra defined by the brackets
\begin{align*}
[e_1,e_4] & = e_2,\;[e_1,e_7]= -2e_1,\;[e_1,e_8]= e_1, \\
[e_2,e_6] & = e_1, \; [e_2,e_7 ]= -e_2,\;[e_2,e_8]= -e_2, \\
[e_4,e_6] & = e_8, \; [e_4,e_7 ] = e_4,\; [e_4,e_8]= -2e_4,\\
[e_6,e_7] & = -e_6, \; [e_6,e_8 ] = 2e_6.
\end{align*}
Then there exists a post-Lie algebra structure on $(\Lg,\Ln)$, given by the following
non-zero products
\begin{align*}
e_3\cdot e_1 & = e_7,\; e_3\cdot e_2= -e_4,\;e_3\cdot e_6= e_5,\;e_3\cdot e_7= -2e_3,\; e_3\cdot e_8= e_3,\\
e_5\cdot e_1 & = -e_6,\; e_5\cdot e_2= e_7+e_8,\;e_5\cdot e_4= e_3,\;e_5\cdot e_7= -e_5,\; e_5\cdot e_8= -e_5.
\end{align*}
\end{ex}
The linear map $\phi$ with $x\cdot y=\{\phi(x),y\}$ is given by 
\[
\phi={\rm diag}(0,0,-1,0,-1,0,0,0).
\]
The construction of this example arises by an easy modification of example $6.3$ in \cite{BU44}.
We have $\Lg=\La+\Lb$ with the abelian subalgebra $\La=\langle e_3,e_5\rangle$ and the subalgebra
$\Lb=\langle e_1,e_2,e_4,e_6,e_7,e_8\rangle$.
The Lie algebra $\Lg$ is not semisimple, because $\Lg^{(1)}=[\Lg,\Lg]\neq \Lg$. We have
$\Lg^{(2)}=[\Lg^{(1)},\Lg^{(1)}]=\Lg^{(1)}$, hence $\Lg$ is not solvable.

\section{Commutative post-Lie algebra structures}

A post-Lie algebra structure on a pair of Lie algebras $(\Lg,\Ln)$ is commutative,
if the algebra product is commutative, i.e., $x\cdot y=y\cdot x$ for all $x,y$. This
means $[x,y]=\{x,y\}$, so that the two Lie algebras are identical. We only write
$\Lg$ instead of the pair $(\Lg,\Lg)$ in this case.

\begin{defi}
A {\it commutative post-Lie algebra structure} on a Lie algebra $\Lg$ 
is a $k$-bilinear product $x\cdot y$ satisfying the identities:
\begin{align}
x\cdot y & =y\cdot x \label{com4}\\
[x,y]\cdot z & = x\cdot (y\cdot z) -y\cdot (x\cdot z)\label{com5} \\
x\cdot [y,z] & = [x\cdot y,z]+[y,x\cdot z] \label{com6}
\end{align}
for all $x,y,z \in V$.
\end{defi}

There is always the {\it trivial} commutative post-Lie algebra structure on $\Lg$, given
by $x\cdot y=0$ for all $x,y\in \Lg$.

\begin{ex}
Let $\Lg$ be an abelian Lie algebra. Then a commutative post-Lie algebra structure on $\Lg$
corresponds to a commutative, associative algebra.
\end{ex}

Using \eqref{com4}, \eqref{com5} and $[x,y]=0$ we have
\[
x\cdot (z\cdot y)=x\cdot(y\cdot z)=y\cdot (x\cdot z) =(x\cdot z)\cdot y
\]
for all $x,y,z\in \Lg$. 

\begin{defi}
Two commutative post-Lie algebra structures $(A,\cdot)$ and $(B,\circ)$ on $\Lg$
are called {\it isomorphic}, if there is a bijective linear map $\phi\colon \Lg\ra \Lg$
satisfying
\begin{align*}
\phi(x\cdot y) & = \phi(x)\circ \phi(y) \\
\phi([x,y]) & = [\phi(x),\phi(y)]
\end{align*} 
for all $x,y \in \Lg$.
\end{defi}

Note that this is in general more than just an isomorphism of the commutative 
algebras $(A,\cdot)$ and $(B,\circ)$.
We require in addition that the algebra isomorphism is also an automorphism of the
Lie algebra $\Lg$. \\
We start with an important result on commutative post-Lie algebra structures.

\begin{prop}\label{5.3}
Any commutative post-Lie algebra structure on a semisimple Lie algebra $\Lg$ is trivial. 
\end{prop}

\begin{proof}
Since $\Lg$ is semisimple, we have $L(x)=\ad (\phi(x))$ for some $\phi\in \End(\Lg)$, such 
that $x\cdot y=[\phi(x),y]$ and
\[
[\phi(x),y]=[\phi(y),x]
\]
for all $x,y\in \Lg$, see Proposition $\ref{2.2}$. It follows $\phi\in \CD(0,1,1)=0$ by
Proposition $5.10$ in \cite{BU44}, or Lemma $6.1$ in \cite{LEL}. 
Hence $x\cdot y=0$ for all $x,y\in \Lg$.
\end{proof}

\begin{cor}\label{5.5}
Let $x\cdot y$ be a commutative  post-Lie algebra structure on a Lie algebra $\Lg$.
Then $\Lg\cdot \Lg\subseteq \rad (\Lg)$.
\end{cor}

\begin{proof} 
Since the left multiplication operators $L(x)$ are derivations of $\Lg$ for all $x\in \Lg$, 
we have $L(x)(\Lr)\subseteq \Lr$, where $\Lr=\rad(\Lg)$ denotes the solvable radical of $\Lg$.
It follows that $\Lg\cdot \Lr \subseteq \Lr$, and of course also $\Lr\cdot \Lg \subseteq \Lr$.
Therefore, there is an induced product $\circ$  on the semisimple Lie algebra $\Ls=\Lg/\Lr$ which is given
by
\[ 
\circ: \Ls \times \Ls \rightarrow \Ls: (x+\Lr, y+\Lr) \mapsto (x+\Lr)\circ (y+\Lr)= x\cdot y +\Lr.
\]
We have $[x,y]_{\Lg}+ \Lr=[x +\Lr,y+ \Lr]_{\Ls}$, and it
is obvious that $\circ$ is a commutative post-Lie algebra structure on $(\Ls,\Ls)$.
By Proposition $\ref{5.3}$, $\circ$ is the trivial zero product.
But this exactly means that $\Lg\cdot \Lg \subseteq \Lr$.
\end{proof}

\begin{cor}
Let $\Lg=\Ls\oplus \Lr$ be a direct Lie algebra sum of a semisimple Lie algebra
$\Ls$ and a solvable Lie algebra $\Lr$. Then for any commutative post Lie structure on $\Lg$ we have
that $\Ls\cdot \Lg=\Lg\cdot \Ls=0$.
\end{cor}
\begin{proof}
Let $x\in \Lg$ and $y,z\in \Ls$, then
\[ 
x \cdot [y,z]= [x \cdot y, z] + [y, x \cdot z]\in [\Lr, \Ls] + [\Ls, \Lr] =0.
\]
As $[\Ls,\Ls]=\Ls$, the result follows.
\end{proof}


In the following we want to generalize Proposition $\ref{5.3}$. For this reason we need to study
derivations of semidirect products of Lie algebras. Let $\Ls$ be a semisimple Lie algebra, 
$\Lr$ a solvable Lie algebra, and consider a representation $\phi\colon \Ls\ra \Der(\Lr)$.
Following \cite{MMO} we define for this section the semidirect product $\Lg=\Lr\rtimes \Ls$ with $\phi$ by
\begin{align}
[(a,x),(b,y)]=([a,b]+\phi(x)b-\phi(y)a, [x,y]) \label{c7}
\end{align}
for all $(a,x),(b,y)\in \Lr\rtimes \Ls$. We have $\Lr=\rad(\Lg)$, so that a derivation $D\in \Der(\Lg)$
satisfies $D(\Lr)\subseteq \Lr$. Hence there exist linear maps $d_1\colon \Lr\ra \Lr$, $d_2\colon \Ls\ra \Ls$
and $f\colon \Ls\ra \Lr$ such that
\begin{align}
D(a,x)=(d_1(a)+f(x), d_2(x)).\label{c8}
\end{align}
The fact that $D$ is a derivation of $\Lg$ now imposes certain conditions on $d_1,d_2$ and $f$.
We have the following lemma, which is similar to Proposition $1.5$ in \cite{MMO}, but has slightly different
assumptions.

\begin{lem}\label{5.7}
The linear map $D\colon \Lg\ra \Lg$ given by $(a,x)\mapsto (d_1(a)+f(x), d_2(x))$ is a derivation of $\Lg$
if and only if the following four conditions are satisfied:
\begin{itemize}
\item[$(a)$] $d_1\in \Der(\Lr)$. 
\item[$(b)$] $d_2\in \Der(\Ls)$.
\item[$(c)$] $f([x,y])=\phi(x)f(y)-\phi(y)f(x)$ for all $x,y\in \Ls$, i.e., $f$ is a $1$-cocycle.
\item[$(d)$] $[d_1,\phi(x)]=\ad_{\Lr}(f(x))+\phi(d_2(x))$ for all $x\in \Ls$. 
\end{itemize}
Here $\ad_{\Lr}$ denotes the adjoint representation of $\Lr$.
\end{lem}

\begin{proof}
The map $D$ is a derivation if and only if
\begin{align}
D([a,b]+\phi(x)b-\phi(y)a,[x,y]) & = [D(a,x),(b,y)]+[(a,x),D(b,y)] \label{c9}
\end{align}
for all $a,b\in \Lr$ and $x,y\in \Ls$. Assume that $D$ is a derivation. Then \eqref{c9} applied to
$a=b=0$ yields
\begin{align*}
(f([x,y]),d_2([x,y]))=(-\phi(y)f(x),[d_2(x),y])+(\phi(x)f(y),[x,d_2(y)]),
\end{align*}
because 
\begin{align*}
D(0,[x,y]) & =(f([x,y]),d_2([x,y])), \\
[D(0,x),(0,y)] & =[(f(x),d_2(x)),(0,y)]=(-\phi(y)f(x),[d_2(x),y]), \\
[(0,x),[D(0,y)] & = [(0,x),(f(y),d_2(y))]=(\phi(x)f(y),[x,d_2(y)]).
\end{align*}
The $\Ls$-component gives $(b)$, and the $\Lr$-component gives $(c)$. In the same way we obtain
$(a)$ from \eqref{c9} with $x=y=0$. Setting $a=y=0$ in \eqref{c9} gives
\[
d_1(\phi(x)b)=[f(x),b]+\phi(d_2(x))b+\phi(x)d_1(b),
\]
which is equivalent to 
\[
(d_1\phi(x)-\phi(x)d_1)(b)=[f(x),b]+\phi(d_2(x))b. 
\]
Since this must hold for all $b\in \Lr$, this implies that
\[
[d_1,\phi(x)]=\ad_{\Lr}(f(x))+\phi(d_2(x)),
\]
which is condition $(d)$. Conversely it is easy to see that the four conditions imply that
$D$ is a derivation.
\end{proof}

We will say that the map $D$ in \eqref{c8} is {\it determined} by the triple $(d_1,f,d_2)$.

\begin{defi}
Denote by
\[
\Der_{\Ls}(\Lr)=\{ d\in \Der(\Lr) \mid \phi(x)d(a)=d(\phi(x)a) \;\forall \, x\in \Ls, a\in \Lr \}
\]
the set of derivations of $\Lr$ which are simultaneously $\Ls$-morphisms.
\end{defi}

\begin{cor}
A triple $(d,0,0)$ determines a derivation of $\Lg$ if and only if $d\in \Der_{\Ls}(\Lr)$.
\end{cor}

\begin{proof}
Condition $(d)$ in Lemma $\ref{5.7}$ just says then that $d$ is an $\Ls$-morphism, while
the other conditions $(a),(b),(c)$ are automatically satisfied when $d\in\Der(\Lr)$.
\end{proof}

In case that $\Lr=\La$ is an abelian Lie algebra, we write $\Der_{\Ls}(\Lr)=\Hom_{\Ls}(\La,\La)=\End_{\Ls}(\La)$.
Any $d\in \End_{\Ls}(\La)$ corresponding to the triple $(d,0,0)$ determines a derivation of $\Lg=\La\rtimes \Ls$.

\begin{lem}\label{innerder}
Let $\Lg=\La\rtimes \Ls$ be the semidirect product of a semisimple Lie algebra $\Ls$ and an abelian Lie algebra
$\La$. Suppose that $d\in \End_{\Ls}(\La)$ determines an inner derivation $D$ of $\Lg$. Then $D=0$ and hence $d=0$.
\end{lem}

\begin{proof}
Suppose that $D=\ad_{\Lg}(a,x)$ for some $(a,x)\in \Lg$, and assume that $x\neq 0$. Then there exists
a $y\in \Ls$ with $[x,y]\neq 0$, so that also $(\ad_{\Lg}(a,x))(0,y)\neq (0,0)$. But this is a contradiction 
since $D(0,y)=(0,0)$ for the triple $(d,0,0)$ by \eqref{c8}. It follows that $x=0$ and $D=\ad_{\Lg}(a,0)$.
We already have  $(\ad_{\Lg}(a,0))(0,y)= D(0,y)=(0,0)$. Since $\La$ is abelian we also have 
$(\ad_{\Lg}(a,0))(b,0)= (0,0)$ for all $b\in \La$. We conclude that $D=\ad_{\Lg}(a,0)$ is the zero map, 
and hence $d=0$.
\end{proof}

The space of outer derivations of $\Lg$ is the space of derivations $\Der(\Lg)=Z^1(\Lg,\Lg)$
modulo the inner derivations $\ad(\Lg)=B^1(\Lg,\Lg)$. We denote this space by ${\rm Out}(\Lg)=H^1(\Lg,\Lg)$. 
The above lemma shows that we have an injective linear map
\[
\psi\colon \End_{\Ls}(\La)\ra H^1(\Lg,\Lg),
\]
mapping an $\Ls$-morphism $d$ to the class of the derivation $D$ modulo inner derivations determined by
the triple $(d,0,0)$.

\begin{prop}\label{5.11}
The map $\psi\colon \End_{\Ls}(\La)\ra H^1(\Lg,\Lg)$ is an isomorphism.
\end{prop}

\begin{proof}
We have to show that the map $\psi$ is onto. Let $D$ be a derivation of $\Lg$ representing
a given class of $ H^1(\Lg,\Lg)$. Then $D$ is determined by a triple $(d_1,f,d_2)$ satisfying
the conditions of Lemma $\ref{5.7}$. Since $\Ls$ is semisimple, the derivation $d_2$ is inner,
so that $d_2=\ad_{\Ls}(x)$ for some $x\in \Ls$. By replacing $D$ with $D-\ad_{\Lg}(x)$ we may assume
that $d_2=0$. Hence $D$ is determined by a triple $(d_1,f,0)$. By condition $(c)$ of  Lemma $\ref{5.7}$
we know that $f\colon \Ls\ra \La$ is a $1$-cocycle. By Whitehead's first Lemma we have $H^1(\Ls,\La)=0$,
since $\Ls$ is semisimple. So $f$ is a $1$-coboundary, and there is a $b\in \La$ such that
$f(x)=-\phi(x)b$ for all $x\in \Ls$. On the other hand we have
\[
\ad_{\Lg}(b,0)(a,x)=[(b,0),(a,x)]=(-\phi(x)b,0)=(f(x),0).
\]
hence by considering $D-\ad_{\Lg}(b,0)$ we may assume that $D$ is determined by a triple $(d_1,0,0)$.
But this just means that the cohomology class of $D$ is equal to $\psi(d_1)$, so that $\psi$ is onto.
\end{proof}


\begin{cor}
Let $\Lg=\La\rtimes \Ls$ be a perfect Lie algebra with nonzero abelian radical $\La$, which
as an $\Ls$-module is irreducible. Assume that $\CD (0,1,1)=0$. Then any commutative post-Lie algebra 
structure on $\Lg$ is trivial.
\end{cor}

\begin{proof}
By Schur's Lemma we have $\End_{\Ls}(\La)=\C \cdot \id$. Using Proposition $\ref{5.11}$ we can write every 
derivation of $\Lg$ as $d+\la \id$ with an inner derivation $d$. Since the left multiplications $L(x)$ 
given by $L(x)(y)=x\cdot y$ are derivations of $\Lg$, there exist linear maps  $\phi\colon \Lg\ra \Lg$ and 
$\la\colon \Lg\ra \C$ such that
\[
L(x)=\ad (\phi(x))+\la(x)\id .
\]
Because $\Lg$ is perfect and $L([x,y])=[L(x),L(y)]$ for all $x,y\in \Lg$ we have
$\tr (L(x))=0$ for all $x\in \Lg$, i.e., $\dim (\La)\cdot \la(x)=0$ for all $x\in \Lg$.
Since $\dim (\La)>0$ it follows that $\la(x)=0$ for all $x\in \Lg$. Then the map $\phi$ 
satisfies, for all $x,y\in \Lg$, 
\[
0=x\cdot y-y\cdot x=[\phi(x),y]+[x,\phi(y)].
\]
This means $\phi\in \CD(0,1,1)=0$, so that $\phi=0$ and $L(x)=0$.
\end{proof}

\begin{rem}
For many Lie algebras $\Lg$ we can show that the space $\CD(0,1,1)$ is trivial, e.g., for
$\Lg=\Ls\Ll_n(\C)\ltimes V(m)$. However, in general this is a problem. Instead we will use
other arguments which do not depend on the space $\CD(0,1,1)$. 
\end{rem}

The aim still is to study commutative post-Lie algebra structures on {\it perfect} Lie algebras.
We can always write a perfect Lie algebra as $\Lg=\Ls\ltimes \Ln$ with Levi subalgebra $\Ls$ and nilpotent
radical $\Ln$. We still need some more lemmas.

\begin{lem}
Let $\Ln$ be a nilpotent Lie algebra and assume that $\Lh$ is a subalgebra of $\Ln$ for which
$\Lh+ [\Ln,\Ln]=\Ln$. Then $\Lh=\Ln$.
\end{lem}
\begin{proof}
For any integer $i\ge 0$, we define $\Lh_i= \Lh+\Ln^i$, where the $\Ln^i$ denote the
terms of the lower central series of $\Ln$, with $\Ln^0=\Ln$. By assumption we have
$\Lh_1= \Lh+ [\Ln,\Ln]=\Ln$ and $\Lh_i=\Lh$ for all $i$ bigger or equal than the
nilpotency class of $\Ln$.\\
It is easy to see that $\Lh_{i+1}$ is an ideal of $\Lh_i$ for any $i$, and that the quotient 
$\Lh_i/\Lh_{i+1}$ is abelian. Now assume that $\Lh\neq \Ln$. Then there exists a positive integer $i$ with
$\Lh_i=\Ln$ and $\Lh_{i+1}\neq \Ln$.
For this $i$ we have that $\Lh_i/\Lh_{i+1}=\Ln/\Lh_{i+1}$ is abelian and hence $[\Ln,\Ln]\subseteq \Lh_{i+1}$.
But then $\Ln= \Lh+[\Ln,\Ln]\subseteq \Lh_{i+1}$ which is a contradiction.
\end{proof}

\begin{lem} \label{ideal}
Suppose that $I$ is an ideal of a perfect Lie algebra $\Lg=\Ls\ltimes \Ln$ with $\Ls \subseteq I$. 
Then it follows that $I=\Lg$.
\end{lem}
\begin{proof}
We first consider the case where $\Ln$ is abelian. Since $I$ contains $\Ls$, we have that
\[ I = \Ls + (I\cap \Ln).\]
Now, $I\cap \Ln$ is also an ideal of $\Lg$ and so $I\cap \Ln$ is an $\Ls$--submodule of $\Ln$. As $\Ls$ 
is semisimple, there exist a complementary $\Ls$-submodule $\La$ of $\Ln$ so that 
$\Ln = (I\cap \Ln)\oplus \La$.  It follows that $\La$ is also an ideal of $\Lg$ and that $\Lg=I \oplus \La$. 
But then $[I,\La]=0$, which gives $[\Lg,\La]=0$. This implies
$\Lg=[\Lg,\Lg]\subseteq [I,I]\subseteq I$, so that $I=\Lg$. \\
Now we consider the general case. When $\Lg=\Ln\rtimes \Ls$ is perfect, then also
$\Lg/[\Ln,\Ln]= \left(\Ln/[\Ln,\Ln] \right) \rtimes \Ls$ is perfect.
From the previous case, we then find that
\[ 
(I+[\Ln,\Ln])/([\Ln,\Ln])=\Lg/[\Ln,\Ln]
\]
and hence $I+[\Ln,\Ln] = \Lg$. It follows that
\[ 
\Lg =\Ls + \Ln = \Ls + (I\cap \Ln) + [\Ln,\Ln].
\]
Hence we obtain $(I \cap \Ln)+ [\Ln,\Ln]= \Ln$. By the previous lemma we have 
$(I\cap \Ln)=\Ln$. From this, it follows that $I=\Lg$.
\end{proof}

Assume that $x\cdot y$ defines a commutative post-Lie structure on the perfect Lie algebra 
$\Lg=\Ls\ltimes \Ln$. Note that $[\Ln,\Ln]$ is both a Lie ideal of $\Lg$ and an ideal for the 
post-Lie product, because all left and right multiplications are derivations of $\Lg$.
It follows that there is an induced post-Lie product on $\Lg/[\Ln,\Ln]$.
We have the following reduction result:
\begin{prop}
Let $x\cdot y$ be a commutative post-Lie product on a perfect Lie algebra $\Lg=\Ls\ltimes \Ln$. If the
induced product on $\Lg/[\Ln,\Ln]$ is the zero product, then also the product on $\Lg$ is zero.
\end{prop}
\begin{proof}
Recall that when $D$ is a derivation of a nilpotent Lie algebra $\Ln$ for which
$D(\Ln)\subseteq [\Ln,\Ln]$, then $D(\Ln^i)\subseteq \Ln^{i+1}$. For the left multiplication
operator $L(x)$ of the commutative post-Lie algebra product $x\cdot y$ on $\Lg$ we have by assumption
that $L(x)(\Lg)\subseteq [ \Ln,\Ln]$. Now, $L(x) $ acts as a derivation on $[\Ln,\Ln]$. Hence we find that
$L^i(x)(\Lg)\subseteq \Ln^{i+1}$ for all $i\ge 1$. It follows that $L(x)$ is a nilpotent map. Hence
$L\colon \Lg \rightarrow \Der(\Lg)$ is a morphism of Lie algebras, where each element of $L(\Lg)$ is a
nilpotent map. By Engel's theorem it follows that $L(\Lg)$ is a nilpotent Lie algebra. But then we must have 
$\Ls \subseteq \ker (L)$. But as the kernel of $L$ is an ideal of $\Lg$ containing $\Ls$, it follows that
$\ker(L)=\Lg$ and so the product $x\cdot y$ is the zero product.
\end{proof}
\begin{cor}
If there exists a perfect Lie algebra with a non-trivial commutative post Lie algebra structure, then 
there also exists such an example with an abelian radical.
\end{cor}

For a perfect Lie algebra $\Lg=\Ls\ltimes \Ln$ denote by $\Der(\Lg,\Ln)$ those derivations 
$D$ of $\Lg$ determined by a triple $(d_1,f,0)$, i.e., those with $D(\Lg)\subseteq \Ln$.

\begin{lem}\label{partder}
Let $\Lg=\Ls\ltimes \La$ be a Lie algebra with abelian radical $\La$. Then $\Der(\Lg,\La)$ 
is a Lie subalgebra of 
$\Der(\Lg)$ which contains $Z^1(\Ls, \La)$ as an abelian ideal and can be written as a semidirect sum
\[  
\Der(\Lg,\La) = Z^1(\Ls,\La) \rtimes \Der_\Ls(\La).
\]
\end{lem}

\begin{proof} 
Since $\La$ is abelian, it follows from Lemma $\ref{5.7}$ that a triple $D=(d,f,0)$ belongs to
$\Der(\Lg,\La)$ if and only if $f\in Z^1(\Ls,\La)$ and $d\in \Der_{\Ls}(\La)$.
Let $D=(d,f,0)$ and $D'=(d',f',0)$ be two derivations in $\Der(\Lg,\La)$.
For any $(a,x)\in \La\rtimes \Ls$ we have that
\begin{eqnarray*}
[D,D'](a,x) & = & D(D'(a,x))- D'(D(a,x))\\
& = & D(d'(a)+f'(x),0) - D'(d(a) + f(x),0)\\
& = & ((dd'-d'd)(a) +(df'-d'f)(x),0)
\end{eqnarray*}
So it follows that the bracket $[D,D']$ is determined by the triple $([d,d'], df'-d'f,0)$. 
The rest of the lemma now follows easily.
\end{proof}

Now we can apply the previous lemmas to commutative post-Lie algebra structures.

\begin{thm}
Let $\Lg=\La\rtimes\Ls$ be a perfect Lie algebra with an abelian radical $\La$. Assume that
$\La$, as an $\Ls$--module, decomposes into pairwise nonisomorphic irreducible $\Ls$--modules. 
Then any commutative post-Lie algebra structure on $\Lg$ is trivial.
\end{thm}

\begin{proof} Let $x\cdot y$ be a commutative post-Lie algebra structure on $\Lg$ with left multiplication
maps $L(x)$. As $\Lg\cdot \Lg\subseteq \La$ by Corollary $\ref{5.5}$, we have that $L(x)\in \Der(\Lg,\La)$. 
Moreover, $\Der_\Ls(\La)=\Hom_\Ls(\La,\La)$ is an abelian Lie algebra, because $\La$ decomposes into
pairwise nonisomorphic irreducible $\Ls$--modules. Hence by Lemma $\ref{partder}$, we have that
$\Der(\Lg,\La)$ is a solvable Lie algebra. But then $L:\Lg \rightarrow \Der(\Lg,\La):x \mapsto L(x)$ 
is a Lie algebra morphism with solvable image. Hence $L(\Ls)=0$, and so $\Ls\subseteq \ker(L)$. 
By Lemma~ $\ref{ideal}$ we can deduce that $\ker(L)=\Lg$ and so the post-Lie product is trivial.
\end{proof}

\begin{thm}
Let $\Lg=\La\rtimes \Ls$ be a perfect Lie algebra with an abelian radical $\La$ and without outer derivations. 
Then any commutative post-Lie algebra structure on $\Lg$ is trivial.
\end{thm}

\begin{proof}
We have that $\Lg=\La\rtimes \Ls$, with $\La$ abelian and $\Ls$ semisimple. Again, we consider a 
commutative post-Lie structure on $\Lg$ and let $L:\Lg \rightarrow \Der(\Lg,\La)$ denote the left 
multiplication, using Corollary $\ref{5.5}$. We claim that $\Der_\Ls(\La)$ is a nilpotent Lie algebra. 
Let $d\in \Der_\Ls(\La)$. Then the triple $(d,0,0)$ determines a derivation of $\Lg$, which is inner 
by assumption. Hence 
\[ 
(d,0,0) = \ad_\Lg(b,x)
\]
for some $b\in \La$ and some $x\in \Ls$. As in the proof of Lemma $\ref{innerder}$ we must have that $x=0$. 
But then
\[ 
(d,0,0)(a,0)=(d(a),0)=\ad_\Lg(b,0)(a,0)= ([b,a],0)= ( \ad_\La(b)(a), 0).  
\]
Hence $d$ is an inner derivation of $\La$. It follows that $\Der_\Ls(\La)$ is a Lie subalgebra of the
nilpotent Lie algebra $\ad_\Ln(\La)$ and by Lemma $\ref{partder}$ we have that $\Der(\Lg,\La)$ is a 
solvable Lie algebra. Hence $L(\Ls)=0$ and as in the previous theorem, we can conclude that $L(\Lg)=0$.
\end{proof}

In view of these results it is natural to formulate the following conjecture:

\begin{con}
Any commutative post-Lie algebra structure on a perfect Lie algebra $\Lg$ is trivial.
\end{con}

Although there is strong evidence in favor of the conjecture, there might be perhaps 
a counterexample in higher dimensions. Cohomological questions on perfect Lie algebras 
are sometimes quite difficult, and it is, for example, already 
not easy to find a perfect Lie algebra without center and without outer derivations, which is not 
semisimple \cite{BEN}.

\section{Commutative structures in low dimensions}

The classification of commutative post-Lie algebra structures on $\Lg$, where $\Lg$ is abelian
corresponds to isomorphism classes of commutative, associative algebras. There are classifications
available in low dimensions over algebraically closed fields, see \cite{BU42}, \cite{POO}. Hence we may assume 
that $\Lg$ is a non-abelian complex Lie algebra.
We may also exclude semisimple Lie algebras, which admit only trivial structures.
The classification in dimension $2$ is as follows, see section $3$ in \cite{BU41}:

\begin{prop}
Let $\Lr_2(\C)$ be the $2$-dimensional non-abelian Lie algebra with Lie bracket $[e_1,e_2]=e_1$.
Every commutative post-Lie algebra structure on $\Lr_2(\C)$ is isomorphic to one of the
structures $A_1$, $A_2$, $A_3$ given by the
left multiplication operators $L(e_i)$ as follows: \\[0.2cm]
\begin{itemize}
\item[(1)] $L(e_1)=0$, $L(e_2)=0$. \\[0.2cm]
\item[(2)] $L(e_1)=0$, $L(e_2)=\begin{pmatrix} 0 & 1  \\ 0 & 0 \end{pmatrix}$. \\[0.2cm]
\item[(3)] $L(e_1)=\begin{pmatrix} 0 & -1 \\ 0 & 0 \end{pmatrix}$,
$L(e_2)=\begin{pmatrix} -1 & 0 \\ 0 & 0 \end{pmatrix}$. \\[0.2cm]
\end{itemize}
Here $A_2$ is an LR-algebra and a pre-Lie algebra, whereas $A_3$ is neither LR nor pre-Lie.
\end{prop}

In dimension $3$, the classification of commutative post-Lie algebra structures on $\Lg$ 
is already rather long, in particular for the infinite family of solvable, non-nilpotent Lie
algebras $\Lr_{3,\la}(\C)$ with brackets $[e_1,e_2]=e_2$ and $[e_1,e_3]=\la e_3$. Therefore we will
only give the classification here for the remaining Lie algebras $\Lg$ of dimension $3$. 
We already have discussed the cases $\Ls\Ll_2(\C)$ and $\C^3$. Hence there only Lie algebras
to consider are $\Lr_{3,1}(\C)$ and $\Ln_3(\C)$:

\begin{prop}
Let $\Lr_{3,1}(\C)$ be the $3$-dimensional solvable non-nilpotent Lie algebra with Lie brackets 
$[e_1,e_2]=e_2$ and $[e_1,e_3]=e_2+e_3$. Every commutative post-Lie algebra structure on 
$\Lr_{3,1}(\C)$ is isomorphic to one of the structures $B_1, B_2$, $B_3$, $B_4$ given by the
left multiplication operators $L(e_i)$ as follows: \\[0.2cm]
\begin{itemize}
\item[(1)] $L(e_1)=0$, $L(e_2)=0$, $L(e_3)=0$. \\[0.2cm]
\item[(2)] $L(e_1)=\begin{pmatrix} 0 & 0 & 0 \\ 1 & 0 & 0 \\ 0 & 0 & 0 \end{pmatrix}$,
$L(e_2)=0$, $L(e_3)=0$. \\[0.2cm]
\item[(3)] $L(e_1)=\begin{pmatrix} 0 & 0 & 0 \\ 0 & 0 & 0 \\ 1 & 0 & 0 \end{pmatrix}$,
$L(e_2)=0$, $L(e_3)=0$. \\[0.2cm]
\item[(4)] $L(e_1)=\begin{pmatrix} 0 & 0 & 0 \\ 0 & 1 & 1 \\ 0 & 0 & 1 \end{pmatrix}$,
$L(e_2)=\begin{pmatrix} 0 & 0 & 0 \\ 1 & 0 & 0 \\ 0 & 0 & 0 \end{pmatrix}$, 
$L(e_2)=\begin{pmatrix} 0 & 0 & 0 \\ 1 & 0 & 0 \\ 1 & 0 & 0 \end{pmatrix}$. \\[0.2cm]
\end{itemize}
\end{prop}

\begin{proof}
A short computation shows that all structures, regardless of isomorphism, can be represented by
\[
L(e_1)=\begin{pmatrix} 0 & 0 & 0 \\ \al & \ga & \frac{\ga}{2\ga-1} \\ \be & 0 & \ga \end{pmatrix},\;
L(e_2)=\begin{pmatrix} 0 & 0 & 0 \\ \ga & 0 & 0 \\ 0 & 0 & 0 \end{pmatrix},\;
L(e_3)=\begin{pmatrix} 0 & 0 & 0 \\ \frac{\ga}{2\ga-1} & 0 & 0 \\ \ga & 0 & 0 \end{pmatrix}.
\]
subject to the condition $\ga(\ga-1)=0$. In particular, we have $2\ga-1\neq 0$.
Let us denote these structures by $B(\al,\be,\ga)$. Using the automorphisms 
\[
\Aut (\Lr_{3,1}(\C))=\left\{ \begin{pmatrix} 1 & 0 & 0 \\ \phi_2 & \phi_5 & \phi_8 \\ \phi_3 & 0 & \phi_5 
\end{pmatrix} \mid \phi_5\neq 0 \right\}
\]
it is easy to see that we obtain $4$ non-isomorphic commutative post-Lie structures
$B(0,0,0)$, $B(1,0,0)$, $B(0,1,0)$, $B(0,0,1)$, which are $B_1, B_2$, $B_3$, $B_4$.
\end{proof}

\begin{prop}\label{6.3}
Let $\Lh_1$ be the $3$-dimensional Heisenberg Lie algebra with Lie bracket $[e_1,e_2]=e_3$.
Every commutative post-Lie algebra structure on $\Lh_1$ is isomorphic to one of the
structures $C_1, C_2(\mu)$, $\mu\in \C$, $C_3$, $C_4$ given by the
left multiplication operators $L(e_i)$ as follows: \\[0.2cm]
\begin{itemize}
\item[(1)] $L(e_1)=0$, $L(e_2)=0$, $L(e_3)=0$. \\[0.2cm]
\item[(2)] $L(e_1)=\begin{pmatrix} 0 & 0 & 0 \\ 1 & 0 & 0 \\ 0 & \mu & 0 \end{pmatrix}$,
$L(e_2)=\begin{pmatrix} 0 & 0 & 0 \\ 0 & 0 & 0 \\ \mu & 0 & 0 \end{pmatrix}$, $L(e_3)=0$. \\[0.2cm]
\item[(3)] $L(e_1)=\begin{pmatrix} 0 & 0 & 0 \\ 0 & 0 & 0 \\ 1 & 0 & 0 \end{pmatrix}$,
$L(e_2)=0$, $L(e_3)=0$. \\[0.2cm]
\item[(4)] $L(e_1)=\begin{pmatrix} 0 & 0 & 0 \\ 0 & 0 & 0 \\ 1 & 0 & 0 \end{pmatrix}$,
$L(e_2)=\begin{pmatrix} 0 & 0 & 0 \\ 0 & 0 & 0 \\ 0 & 1 & 0 \end{pmatrix}$, $L(e_3)=0$.\\[0.2cm]
\end{itemize}
We have $C_2(\mu)\cong C_2(\nu)$ if and only if $\mu=\nu$.
\end{prop}

\begin{proof}
A short computation shows that all structures, regardless of isomorphism, can be represented by
\[
L(e_1)=\begin{pmatrix} \al & \de & 0 \\ \be & -\al & 0 \\ \ga & \ep & 0 \end{pmatrix},\;
L(e_2)=\begin{pmatrix} \de & \ka & 0 \\ -\al & -\de & 0 \\ \ep & \la & 0 \end{pmatrix},\;
L(e_3)=0,
\]
subject to the polynomial conditions
\begin{align*}
\al\de+\be\ka & = 0,\\
\al\ka-\de^2 & = 0, \\
\al^2+\be\de & = 0,\\
\ga\de-2\al\ep-\be\la & = 0,\\
\al\la+\ga\ka-2\de\ep & =0.
\end{align*}
Note that the first three equations ensure that all $L(x)$ are nilpotent. In case that $\al, \be$ or $\ka$
is nonzero, using the automorphisms 
\[
\Aut (\Lh_1(\C))=\left\{ \begin{pmatrix} \phi_1 & \phi_4 & 0 \\ \phi_2 & \phi_5 & 0 \\ \phi_3 & \phi_6 & 
\phi_1\phi_5 -\phi_2\phi_4 \end{pmatrix} \mid \phi_1\phi_5-\phi_2\phi_4 \neq 0 \right\}
\]
it is straightforward to see that all structures are isomorphic to some $C_2(\mu)$.
Otherwise we have $\al=\be=\ka=0$, and products $e_1\cdot e_1=\ga e_3$, $e_1\cdot e_2=\ep e_3$ and
$e_2\cdot e_2=\la e_3$. Up to post-Lie structure isomorphism, we obtain three different structures,
represented by $(\ga,\ep,\la)=(1,0,0), (1,0,1),(0,0,0)$.
\end{proof}

\section{Commutative post-Lie algebra structures on nilpotent Lie algebras}

As we have seen in Proposition $\ref{6.3}$, the left multiplication operators $L(x)$ are nilpotent
for any commutative post-Lie algebra structure on the Heisenberg Lie algebra $\Lh_1$.
We can extend this observation to all non-abelian $2$-generated nilpotent Lie algebras. This includes
filiform nilpotent Lie algebras.

\begin{prop}\label{7.1}
Let $\Ln$ be a non-abelian nilpotent Lie algebra which is generated by two elements. Then for any
commutative post-Lie algebra structure on $\Ln$ the left multiplication operators $L(x)$ are nilpotent.
\end{prop}

\begin{proof}
The product on $\Ln$ induces a commutative post-Lie structure on $\Ln/\Ln^2$.
As $\Ln/\Ln^2$ is the $3$-dimensional Heisenberg Lie algebra, we know by the previous discussion that
the induced left multiplication maps are nilpotent. This means that for any $x\in \Ln$, the map $L(x)\in \Der(\Ln)$
is a derivation which induces a nilpotent map on $\Ln/\Ln^2$, and hence also on $\Ln/\Ln^1$, since we have
$$
(\Ln/\Ln^1)/(\Ln/\Ln^2)\cong \Ln^1/\Ln^2.
$$
Since any derivation of $\Ln$ which induces a nilpotent derivation on $\Ln/\Ln^1$ is itself nilpotent,
we can conclude that indeed all maps $L(x)$ are nilpotent.
\end{proof}

\begin{cor}
Let $x\cdot y$ be a commutative post-Lie algebra structure on a complex nilpotent Lie algebra
of dimension $n\le 4$ without abelian factor. Then all left multiplication maps $L(x)$ are nilpotent.
\end{cor}

We cannot extend the result to nilpotent Lie algebras with an abelian factor:

\begin{ex}
Let $\Ln=\Lh_1\oplus \C$ be the direct sum of the Heisenberg Lie algebra $\Lh_1$ and 
the abelian Lie algebra $\C$. Let $(e_1,\ldots ,e_4)$ be a basis of $\Ln$ with $[e_1,e_2]=e_3$.
Then 
\[
e_1\cdot e_1=e_1\cdot e_4=e_4\cdot e_1=e_4\cdot e_4=e_4
\]
defines a commutative post-Lie algebra structure on $\Ln$, where not all maps $L(x)$
are nilpotent.
\end{ex}

It is also clear that the result holds for all characteristically nilpotent Lie algebras,
because the derivation algebra of such Lie algebras is nilpotent. The above results and some more computational
evidence leads to the following question in general:

\begin{qu}
Is it true that all left multiplication maps are nilpotent for every
commutative post-Lie algebra structure on $\Ln$, where $\Ln$ is a nilpotent Lie algebra 
without abelian factor ?
\end{qu}

\end{document}